\documentclass[10pt]{article}
\font\smallit=cmti10
\font\smalltt=cmtt10

\usepackage{amssymb,latexsym,amsmath,epsfig,amsthm}

\def\Q{\ensuremath {\mathbb{Q}}}
\def\C{\ensuremath {\mathbb{ C}}}
\def\Z{\ensuremath {\mathbb{Z}}}

\def\g{\ensuremath {\gcd}}
\def\gb{\ensuremath {\gcd_b}}
\def\l{\ensuremath {{\Lambda}}}

\usepackage{mathtools}

\DeclarePairedDelimiter\floor{\lfloor}{\rfloor}

\newcommand\be{\begin{equation}}
\newcommand\ee{\end{equation}}

\makeatletter

\renewcommand\section{\@startsection {section}{1}{\z@}
	{-30pt \@plus -1ex \@minus -.2ex}
	{2.3ex \@plus.2ex}
	{\normalfont\normalsize\bfseries\boldmath}}

\renewcommand\subsection{\@startsection{subsection}{2}{\z@}
	{-3.25ex\@plus -1ex \@minus -.2ex}
	{1.5ex \@plus .2ex}
	{\normalfont\normalsize\bfseries\boldmath}}

\renewcommand{\@seccntformat}[1]{\csname the#1\endcsname. }

\makeatother

\newtheorem{theorem}{Theorem}

\newtheorem{corollary}{Corollary}
\theoremstyle{definition}
\newtheorem{definition}[theorem]{Definition}

\theoremstyle{remark}
\newtheorem{remark}[theorem]{\bf Remark}

\begin{document}
	
\begin{center}
	\uppercase{\bf The distribution of the generalized greatest common divisor and visibility of lattice points}
	\vskip 20pt
	{\bf Jorge Fl\'orez}\\
	{\smallit Department of Mathematics, Borough of Manhattan Community College, City University of New York, 199 Chambers Street, New York, NY 10007, USA}\\
	{\tt jflorez@bmcc.cuny.edu}\\ 
	\vskip 10pt
	{\bf Cihan Karabulut\footnote{The author was partly supported by Assigned Release Time (ART) program for research from William Paterson University.}}\\
	{\smallit Department of Mathematics, William Paterson University, New Jersey 07470, USA}\\
	{\tt karabulutc@wpunj.edu}\\
	\vskip 10pt
	{\bf Elkin Quintero Vanegas}\footnote{The author was partly supported by the Coordena\c{c}\~{a}o de Aperfei\c{c}oamento de Pessoal de N\'{i}vel Superior - Brasil (CAPES) - Finance Code 001.}\\
	{\smallit Departamento de Matem\'atica, Universidade Federal do Amazonas, Av. General Rodrigo Oct\'avio, 6200, Coroado I, Setor Norte, Manaus - AM,  69080-900, Brazil}\\
	{\tt eoquinterov@ufam.edu.br}
\end{center}
\vskip 20pt
\centerline{\smallit Received: , Revised: , Accepted: , Published: } 
\vskip 30pt

\centerline{\bf Abstract}
\noindent
For a fixed $b\in \mathbb{N}=\{1,2,3,\dots\}$, Goins et al. \cite{Harris} defined the concept of $b$-visibility for a lattice  point $(r,s)$ in $L=\mathbb{N}\times \mathbb{N}$  which states that $(r,s)$ is $b$-visible from the origin if it lies on the graph of $f(x)=ax^b$, for some positive $a\in \Q$, and no other lattice point in $L$ lies on this graph between $(0,0)$ and $(r,s)$. Furthermore, to study the density of $b$-visible points in $L$ Goins et al.  defined a generalization of greatest common divisor, denoted by $\gcd_b$, and proved that the proportion of $b$-visible lattice points in $L$ is given by $1/\zeta(b+1)$, where $\zeta(s)$ is the Riemann zeta function. In this paper we study the mean values of arithmetic functions $\Lambda:L\to \mathbb{ C}$ defined using $\gcd_b$ and recover the main result of \cite{Harris} as a consequence of the more general results of this paper. We also investigate a generalization of a result in \cite{Harris}   that asserts that there are arbitrarily large rectangular arrangements of $b$-visible points in the lattice $L$ for a fixed $b$, more specifically, we give necessary and sufficient conditions for an arbitrary rectangular arrangement  containing $b$-visible and $b$-invisible points to be realizable in the lattice $L$.
Our result is inspired by the work of Herzog and Stewart \cite{Herzog} who proved this in the case $b=1$.

\pagestyle{myheadings}
\markright{\smalltt INTEGERS: 19 (2019)\hfill}
\thispagestyle{empty}
\baselineskip=12.875pt
\vskip 30pt

\section{Introduction}

Let $L$ denote the lattice $\mathbb{N}\times \mathbb{N}$. A point $(r,s)$ in $L$ is called \textit{visible from the origin}, or simply \textit{visible}, if $\g(r,s)=1$, which is equivalent to having no other integer lattice points on the line segment joining the point $(0,0)$ and the point $(r,s)$. 

A classical result which predates the Prime Number Theorem asserts that the proportion of visible points in $L$  is given by $1/\zeta(2)=6/\pi^2\approx 0.608$, where $\zeta(s)$ is the Riemann zeta function. In a recent paper \cite{Harris}, Goins et al. explore the visibility of lattice points on \textit{generalized lines of sight}. Here, by generalized line of sight we mean that the line from the origin to the lattice point $(r,s)$ is no longer a straight line segment but a more general curve. In particular, they study the density of $b$-visible points from the origin which are the points $(r,s)$ in $L$ that lie on the graph of $f(x)=ax^b$ where $a$ is a rational number and $b$ is a positive integer and no other point in $L$ lies on this curve (i.e., line of sight) between $(0,0)$ and $(r,s)$. Remarkably, they show (cf. \cite[Theorem 1]{Harris}) that the proportion of $b$-visible points in $L$ is given by $1/\zeta(b+1)$.

To study the density of $b$-visible points, they develop a generalization of the greatest common divisor.

\begin{definition} 
	Let $b\in \mathbb{N}$. The \textit{generalized greatest common divisor} of $r$ and $s$ with respect to $b$ is denoted by $\text{gcd}_b$ and is defined by
	\[
	\text{gcd}_b(r,s):=\max \{\,k\in \mathbb{N}\,|\, \text{$k$ divides $r$ and $k^b$ divides $s$}\,\}.
	\]
\end{definition}

Notice that when $b=1$, $\text{gcd}_b$ coincides with the classical greatest common divisor and one immediately recovers the classical result mentioned earlier pertaining to the proportion of visible points in $L$. Moreover, it is shown in \cite{Harris} that a point $(r,s)\in L$ is $b$-visible if and only if $\text{gcd}_b(r,s)=1$. 

In this work, we first begin by studying the mean values of arithmetic functions defined in terms of the $\gb$. That is, for a fixed $b\in \mathbb{N}$ and an arithmetic function $f:\mathbb{N}\to \C$, we define $\l_f:L\to \C$ to be 
\begin{equation}\label{i-ray-fun}
\l_f(r,s):=f(\text{gcd}_b(r,s)).
\end{equation}

We let $M(\l_f)$ denote the mean value of $\l_f$ over $L$ (see Section \ref{stats} for the precise definition) and $\zeta_f(s)=\sum f(n)\,n^{-s}$ denote the Dirichlet series associated to $f$. Then our first result is as follows.

\begin{theorem}\label{Intro-Mean-value}
	Fix $b\in \mathbb{N}$ and let $f:\mathbb{N}\to \C$ be some arithmetic function satisfying
	\begin{equation}\label{condition-mean-val}
	\frac{1}{N}\sum_{k=1}^{N}\frac{\big|(f\ast \mu)(k)\big|}{k}\to 0 \quad (N\to \infty), 
	\end{equation}
	where $\mu$ is the Mobius function and $*$ is the Dirichlet convolution.
	Then $M(\l_f)$ exists and
	\begin{equation}\label{me-val}
	M(\l_f)=\frac{\zeta_f(b+1)}{\zeta(b+1)},
	\end{equation}
	as long as $\zeta_{f}(s)$ is absolutely convergent at $s=b+1$. 
	Moreover,  the condition  \eqref{condition-mean-val}  holds, for example, when $f$ is a bounded function.
\end{theorem}

\begin{remark}\label{harris-theo}
	We can recover the main result of \cite{Harris} as an immediate application of Theorem \ref{Intro-Mean-value} by letting $f(n)=\floor*{\frac{1}{n}}.$ 
\end{remark}

As another consequence of Theorem \ref{Intro-Mean-value} we have the following result which gives the density of the points $(r,s) \in L$ with a fixed $\gb$. 

\begin{theorem}\label{intro-gcdb-density}
	Fix two positive integers $b$ and $k$. Then  the proportion of points $(r,s)\in L$ for which 
	$\text{gcd}_b(r,s)=k$ is
	\[
	\frac{1}{k^{b+1}\zeta(b+1)}.
	\] 
\end{theorem} 

We also study the average value of $\gb$ and obtain the following asymptotic formula.

\begin{theorem}\label{intro-avg-gcdb}
	Fix $b\in\mathbb{N}$ with $b\geq 2$. Then  
	\[
	\sum_{\substack{ 0<r\leq x\\ 0<s\leq x^b}} \normalfont \text{gcd}_b(r,s)=x^{b+1}\frac{ \zeta(b)}{\zeta(b+1)}+O(E(x)),
	\]
	where
	\[
	E(x)=
	\begin{cases}
	x^2\, \log x \quad &(b=2)\\
	x^{b} \quad &(b>2).
	\end{cases}
	\]
\end{theorem}

In the last part of this work, we explore a generalization of a result of Goins et al. \cite[Theorem~2]{Harris} that asserts that there are arbitrarily large rectangular arrangements in $L$ consisting  only of $b$-invisible points. More specifically, given an arbitrary rectangular arrangement consisting  of $b$-visible and $b$-invisible points, which we call a  $b$-\textit{pattern} (see Definition \ref{b-pattern}), we provide necessary and sufficient conditions for it to be realizable in the lattice $L$. This generalization is motivated by the work of Herzog and Stewart, who in \cite[ Theorem~1]{Herzog} have completely characterized the conditions for a given pattern (in the case $b=1$) consisting of visible and invisible points to be realizable in $L$. In particular, they showed that the lattice $L$ contains arbitrarily large rectangular patches consisting entirely of invisible points. The following theorem which we prove in Section \ref{graph} generalizes Theorem 1 in \cite{Herzog} to our setting and completely characterizes the conditions for a given $b$-pattern  to be realizable in $L$.

Before stating the theorem we need to introduce the following definition. Let $m$ be a positive integer and $S$ be any collection of $m^{b+1}$ points in $L$. We say that $S$ is a \textit{complete rectangle modulo}  ${(m,m^b)}$ if it contains a complete system of residues of the Cartesian product $\Z/m\Z \times \Z/m^b\,\Z$. 

\begin{theorem}\label{intro-pattern}
	For a fixed $b>1$ let $P$ be a $b$-pattern consisting of $b$-visible and $b$-invisible points. Then $P$ is realizable in $L$ if and only if the set  of $b$-visible points in $P$ fails to contain a complete rectangle modulo $(p,p^b)$ for every prime $p$. 
\end{theorem}

In Section \ref{graph}, we also give a number of immediate corollaries of this theorem which state whether or not certain $b$-patterns $P$ can be realizable in $L$. Indeed, as a corollary we recover Theorem 2 in \cite{Harris}: 

\begin{corollary}
	$L$ contains arbitrarily large rectangular patches consisting entirely of $b$-invisible points. 
\end{corollary}

Our paper is organized as follows. Section \ref{stats} contains the necessary definitions and the proofs of Theorem \ref{Intro-Mean-value}, Theorem \ref{intro-gcdb-density} and Theorem \ref{intro-avg-gcdb}. Section \ref{graph} provides a proof of Theorem \ref{intro-pattern} and discusses various consequences of this theorem. 
\section{Distribution of $\normalfont \text{gcd}_b$}\label{stats}
\subsection{Mean value of arithmetic functions of generalized greatest common divisor}\label{Mean values of arithmetic functions}

For a positive integer $N$ let 
\begin{equation}\label{rectangle}
T_N:=\{(r,s)\in L\,|\,0< r, s\leq N\}.
\end{equation}
We define the mean value of a function $\Lambda:L\to \C$ to be the limit
\begin{equation}\label{mean-value-definition}
M(\l):=\lim_{N\to \infty} \frac{\sum_{(r,s)\in T_N}\l(r,s)}{|T_N|}.
\end{equation}

For an arithmetic function $f:\mathbb{N}\to \C$ we will be interested in functions $\l_f:L\to \C$ that are as in Definition \ref{i-ray-fun}. For these functions Theorem \ref{Intro-Mean-value} shows that $M(\l_f)$ can be computed  in terms of the Dirichlet series $\zeta_f(s)=\sum f(n)\,n^{-s}$ associated to $f$ and the Riemann zeta function $\zeta(s)$. When $b=1$, the computation of  $M(\l_f)$  was previously considered in  \cite[Theorem~7]{ushiroya2012mean}. We now present a proof of the theorem.

\begin{proof}[Proof of Theorem \ref{Intro-Mean-value}]

	Let $u$ be the constant function $u(n)=1$ for all $n\in \mathbb{N}$. Let $g$ denote the Dirichlet convolution $f\ast \mu$ of $f$ with the M\"{o}bius function $\mu$ so that $g\ast u=f$. Then $\l_f=\l_{g\ast u}$ and the number of times the term $g(k)$  (for  a given $k\in \mathbb{N}$) appears in the sum
	\[
	q_N:=\sum_{0< r,s\leq N}\l_{g\ast u}(r,s)=\sum_{0< r,s\leq N}\left( \,\sum_{k\,\big|\,\text{gcd}_b(r,s)} g(k)\,\right)
	\]
	is $\floor*{\frac{N}{k}}\floor*{\frac{N}{k^b}}$. Therefore
	\[
	q_N=\sum_{k=1}^Ng(k)\,\floor*{\frac{N}{k}}\floor*{\frac{N}{k^b}}.
	\]
	
	On the other hand,  
	\begin{align}\label{ineq}
	\frac{N^2}{k^{b+1}}-\floor*{\frac{N}{k}}\floor*{\frac{N}{k^b}}
	&= 
	\frac{N}{k^{b}}
	\left(\frac{N}{k}-\floor*{\frac{N}{k}}\right)
	+\floor*{\frac{N}{k}}\left(\frac{N}{k^b}-\floor*{\frac{N}{k^b}}\right) \\
	&\leq \frac{N}{k^b}+ \floor*{\frac{N}{k}}\leq \frac{2N}{k}\nonumber.
	\end{align}
	Then by our hypothesis on $g=f\ast \mu$ we have 
	\[
	\frac{|q_N-\sum_{k=1}^N g(k)\frac{N^2}{k^{b+1}}|}{N^2}\leq \frac{2\,N\,H_N}{N^2}\to 0,
	\]
	where $H_N=\sum_{k=1}^N|g(k)|k^{-1}$. This implies
	\[
	M(\l_f)\to \zeta_g(b+1).
	\]
	
	The equality \eqref{me-val} follows now from the fact that $\zeta_f(s)=\zeta_g(s)\zeta(s)$ for every $s$ for which $\zeta_f(s)$ and $\zeta(s)$ are absolutely convergent.
\end{proof}

As another consequence of  Theorem \ref{Intro-Mean-value} we can also count the proportion of lattice points with a given $\text{gcd}_b$. More specifically, fix two positive integers $b$ and $k$. For $N>0$ let $T_N$ be the set defined in  \eqref{rectangle}, then the proportion of lattice points with $\text{gcd}_b$ equal to $k$ is defined by the limit
\[
\lim_{N\to \infty} 
\frac{ 
	\big|\{(r,s)\in T_N\, | \, \text{gcd}_b(r,s)=k\}\big|
}{
	|T_N|}.
\]
The value of this limit is given in Theorem \ref{intro-gcdb-density}, whose proof follows from Theorem \ref{Intro-Mean-value} by taking $f:\mathbb{N}\to \C$ to be the function  defined by $f(n)=k$ for $n=k$, and $f(n)=0$ for $n\neq k$.

More generally,  we obtain the following generalization to $b\geq 1$ of a result of Cohen   \cite[Corollary 3.2]{cohen1959arithmetical}.

\begin{corollary}Let $S$ be  a subset of $\mathbb{N}$. Then the proportion of lattice points $(r,s)\in L$
	for which $\gcd_b(r,s)\in S$ is given by $\zeta_S(1+b)/\zeta(b+1)$,
	where 
	\[
	\zeta_S(b+1)=\sum_{k=1, k\in S}^{\infty}\frac{1}{k^{b+1}}.
	\]
	More precisely,
	\[
	\lim_{N\to\infty}\frac{|\{(r,s)\in L\,|\,0<r,s\leq N,\,\text{gcd}_b(r,s)\in S\}|}{N^2}=\frac{\zeta_S(b+1)}{\zeta(b+1)}.
	\]
\end{corollary}

\subsection{The average value of general arithmetic functions in the lattice}
For a general function $\Lambda:L\to \C$ we can still give a description of the mean value $M(\Lambda)$ in terms of a Dirichlet series whose $k$-th coefficient ( $k>0$ ) is the average value of $\Lambda$ on the points with $\text{gcd}_b=k$ for a fixed $b\in \mathbb{N}$. More specifically, we define
\[
\zeta_{\Lambda,b}(s):=\sum_{k=1}^{\infty} \frac{M_{b,k}(\Lambda)}{k^s},
\]
where the coefficient $M_{b,k}(\Lambda)$, $k>0$, is the average value of $\Lambda$ on the points having $\text{gcd}_b=k$, i.e.,
\[
M_{b,k}(\Lambda):=\lim_{N\to \infty} 
\frac 
{\sum_{(r,s)\in T_{N,b,k}} \Lambda(r,s)}
{|T_{N,b,k}|},
\]
and
\begin{equation}\label{partial-sum-k}
T_{N,b,k}:=\{\,(r,s)\in L\,|\,0<r,s\leq N\, ,\, \text{gcd}_b(r,s)=k\, \}.
\end{equation}

As an interesting note, we can give a  formulation of $M_{b,k}(\Lambda)$ in geometric terms if we interpret the $\text{gcd}_b$ as a metric as follows. Given a point $A=(r,s)$ in $L'=L\,\cup\, \{(0,0)\}$ we let  $||A||_b:=\text{gcd}_b(r,s)$ if $A\neq (0,0)$, and $||A||_b=0$ if $A=(0,0)$. We say that two nonzero points $A=(r_1,s_1)$ and $B=(r_2,s_2)$ in $L$ are in the same $b$-curve of vision if
\[
\left(\frac{r_1}{||A||_b},\frac{s_1}{||A||_b^b}\right)
=
\left(\frac{r_2}{||B||_b},\frac{s_2}{||B||_b^b}\right),
\]
i.e., the points $A$ and $B$ both lie on the graph of $f(x)=ax^b$, for some positive rational number $a$. 

For $A,B\in L'$ we define the metric
\[
d_b(A,B)
:=\begin{cases}
\big|\,||B||_b-||A||_b\,\big|,\ \text{if $A$ and $B$ are in the same $b$-curve of vision,}\\
||A||_b+||B||_b,\ \text{otherwise.}
\end{cases}
\]
In particular, $d_b(O,A)=||A||_b$.

With this definition of metric, the ball centered at the origin having radius 1 is exactly the set of $b$-visible points from the origin. Moreover, the set of points whose $\text{gcd}_b$  is a fixed integer $k$ can be thought of as sphere of radius $k$ centered at the origin:  
\begin{equation}\label{sphere}
S_k^b:=\{\,A\in L\,:\, ||A||_b=k\,\}.
\end{equation}
Furthermore, according to Theorem \ref{intro-gcdb-density} the sphere $S_k^b$ has density  $1/(k^{b+1}\zeta(b+1))$.

With this notation, $M_{b,k}(\Lambda)$ can be thought of  as the average value of $\Lambda$ on the sphere $S_{k}^b$.

The following theorem informs us on how to calculate $M(\Lambda)$ from $\zeta_{\Lambda,b}$. We remark that \cite{ushiroya2012mean} also has a description of $M(\Lambda)$ but in terms of  certain multiple Dirichlet series.

\begin{theorem}
	Fix $b\in \mathbb{N}$ and let $\Lambda:L \to \C$ be a bounded function. 
	Then $\zeta_{\Lambda,b}(s)$ is convergent at $s=b+1$ and
	\begin{equation}\label{average-2}
	M(\Lambda)=\frac{\zeta_{\Lambda,b}(b+1)}{\zeta(b+1)}.
	\end{equation}
\end{theorem}
\begin{proof}
	We begin by showing that
	\begin{equation}\label{average-k-points}
	M(\Lambda)=\sum_{k=1}^\infty M_k(\Lambda),
	\end{equation}
	where $M_k(\Lambda)$ is defined as the limit
	\[
	\lim_{N\to \infty}\frac{\sum_{(r,s)\in T_{N,b,k}} \Lambda(r,s)}{N^2}
	\]
	and $T_{N,b,k}$ is  as in \eqref{partial-sum-k}. 
	
	In order to show this, we start with dividing the following identity by $N^2$ 
	\[
	\sum_{0<r,s\leq N}\Lambda(r,s)=\sum_{k=1}^{\infty}\, \sum_{\substack{0<r,s\leq N\\ \text{gcd}_b(r,s)=k}} \Lambda(r,s),
	\]
	thus obtaining
	\begin{equation}\label{k-points}
	\frac{\sum_{0<r,s\leq N}\Lambda(r,s)}{N^2}=\sum_{k=1}^{\infty} \frac{S_{N,k}}{N^2},
	\end{equation}
	where
	\[
	S_{N,k}=\sum_{\substack{0<r,s\leq N\\ \text{gcd}_b(r,s)=k}} \Lambda(r,s).
	\]
	Let $C>0$ such that $|\Lambda (r,s)|\leq C$ for all $(r,s)\in L$. Then
	\[
	|S_{N,k}|\leq q_{N,b,k}\,C,
	\]
	where $q_{N,b,k}=|T_{N,b,k}|$. Using the trivial bound 
	\[
	q_{N,k}\leq \floor*{\frac{N}{k}}\,\floor*{\frac{N}{k^b}}
	\]
	we obtain the estimate
	\[
	\frac{|S_{N,k}|}{N^2}\leq \frac{C}{k^{b+1}},
	\]
	for all $N\geq 1$. The Weirstrass $M$-test allows us now to interchange the limit $N\to \infty$ in the infinite sum \eqref{k-points}, thus  obtaining \eqref{average-k-points}.
	
	Finally, \eqref{average-2} is a consequence of the identity
	\[
	M_k(\Lambda)=\frac{M_{b,k}(\Lambda)}{k^{b+1}\,\zeta(b+1)},
	\]
	which in turn follows by taking the limit as $N\to \infty$ in
	\[
	\frac{S_{N,k}}{N^2}=\frac{S_{N,k}}{q_{N,b,k}}\,\frac{q_{N,b,k}}{N^2},
	\]
	and observing that $q_{N,b,k}/N^2\to 1/(k^{b+1}\zeta(b+1))$, by virtue of Theorem \ref{intro-gcdb-density}.
\end{proof}

Clearly this theorem immediately implies Theorem  \ref{Intro-Mean-value} as $M_{b,k}(\Lambda_f)=f(k)$, for all $k>0$, for $\Lambda_f$ defined  as in \eqref{i-ray-fun}.

\subsection{The average value of $\text{gcd}_b$}

In this section we are going to study the average value of the $\text{gcd}_b$ throughout the points of the lattice $L$ in more detail. The case $b=1$ has been previously considered in the paper \cite{Cohen} and 
asserts that
\begin{equation}\label{cohen}
\sum_{r,s\leq x} \normalfont \text{gcd}(r,s)=\frac{ x^2}{\zeta(2)}\left( \log x+2\gamma-\frac{1}{2}-\frac{\zeta'(2)}{\zeta(2)}\right)+O(x^{1+\theta+\epsilon})
\end{equation}
for every $\epsilon>0$, where $\gamma$ is the Euler constant and $\theta$ is the exponent appearing in Dirichlet's divisor problem, namely, $\theta$ is the smallest positive number  such that for every $\epsilon>0$
\[
\sum_{n\leq x}\tau(n)=x\log x+(2\gamma-1)x+O(x^{\theta+\epsilon}).
\]
Moreover, it is known that $1/4\leq \theta \leq 131/416$, where the upper bound  was found by  Huxley in \cite{Huxley} and it is the  best upper bound for $\theta$ up to date.

The case $b\geq 2$ is treated in Theorem \ref{intro-avg-gcdb}  whose proof we give below. We highlight that, unlike  \eqref{cohen}, Theorem \ref{intro-avg-gcdb} does not provide  secondary error terms.
We also point out that even tough it is classical to consider the sum
\[
\sum_{\substack{ 0<r\leq x\\ 0<s\leq x}}  \text{gcd}_b(r,s),
\]
it is more natural in this context to rather work with the sum
\[
\sum_{\substack{ 0<r\leq x\\ 0<s\leq x^b}}  \text{gcd}_b(r,s).
\]
\begin{proof}[Proof of Theorem \ref{intro-avg-gcdb}]
	Let $b\geq 2$. From the classical identity $\sum_{d|n}\phi(d)= n$, $n\geq 1$ for the Euler totient function $\phi$, we obtain the following identities
	\begin{align*}
	\sum_{\substack{ 0<r\leq x\\ 0<s\leq x^b}}  \text{gcd}_b(r,s)
	&=\sum_{\substack{ 0<r\leq x\\ 0<s\leq x^b}}\,\sum_{d\,|\,\text{gcd}_b(r,s)} \phi(d)\\
	&=\sum_{d\leq x}\,\phi(d)\,\floor*{\frac{x}{d}}\,\floor*{\frac{x^b}{d^b}}\\
	&=\sum_{d\leq x}\,\phi(d)\,\left\lbrace \frac{x^{b+1}}{d^{b+1}}
	+O\left(\frac{x^{b}}{d^{b}}\right)\right\rbrace \\
	&= x^{b+1}\sum_{d\leq x} \frac{\phi(d)}{d^{\,b+1}} 
	+O \left(x^b \sum_{d\leq x} \frac{\phi(d)}{d^{b}} \right).                 
	\end{align*}
	
	We now invoke the following estimates (cf. \cite{Apostol} Chapter 3, Exercises 6 and 7)
	\begin{equation}\label{phi1}
	\sum_{n\leq x} \frac{\phi(n)}{n^2}=\frac{1}{\zeta(2)}\,\log x
	+\frac{\gamma}{\zeta(2)}-A+O\bigg(\frac{\log x}{x}\bigg),
	\end{equation}
	where $A=\sum_{n=1}^{\infty}\frac{\mu(n)\log n}{n^2}=\frac{\zeta'(2)}{\zeta(2)^2}$, and
	\begin{equation}\label{phi2}
	\sum_{n\leq x} \frac{\phi(n)}{n^\alpha}=\frac{\zeta(\alpha-1)}{\zeta(\alpha)}
	+\frac{x^{2-\alpha}}{2-\alpha}\frac{1}{\zeta(2)}+O(x^{1-\alpha}\log x),
	\end{equation}
	where $\alpha>1$ and $\alpha\neq 2$.
	
	Applying \eqref{phi1}  and \eqref{phi2} to $b=2$ we have
	\begin{align*}
	\sum_{\substack{ r\leq x\\ s\leq x^2}} \normalfont \text{gcd}_2(r,s)
	&=
	x^3\left\lbrace \frac{\zeta(2)}{\zeta(3)}
	-\frac{1}{x}\frac{1}{\zeta(2)}+O\left(\frac{\log x}{x^2}\right) \right\rbrace +O(x^2\log x)\\
	&=x^3\frac{\zeta(2)}{\zeta(3)}+O(x^2\,\log x),
	\end{align*}
	and applied to $b\geq 3$ we similarly obtain
	\begin{align*}
	\sum_{\substack{ r\leq x\\ s\leq x^b}} \normalfont \text{gcd}_b(r,s)
	&=
	x^{b+1}\left\lbrace \frac{\zeta(b)}{\zeta(b+1)}
	+\frac{x^{1-b}}{1-b}\frac{1}{\zeta(2)}+O(x^{-b}\log x) \right\rbrace +O(x^b)\\
	&=x^{b+1}\frac{\zeta(b)}{\zeta(b+1)}+O(x^{b}),
	\end{align*}
	which proves the theorem.
\end{proof}

\section{The graph of $b$-visible points}\label{graph}

The collection of all $b$-visible points in the lattice $L$ can be thought of as a graph,  denoted by $G_b$, 
if we build an edge between  two given $b$-visible points whenever the Euclidean distance 
between them is 1.  In this section 
we will prove some results concerning the connectivity of the graph $G_b$.

We start with the following result which states that $G_b$ is on average  $4/\zeta(b+1)$-connected, i.e., every point in the graph $G_b$ is on average connected to $4/\zeta(b+1)$ points.

\begin{theorem}
	For an arbitrary point in the lattice $L$, there are on average
	\[
	\frac{4}{\zeta(b+1)}
	\]
	$b$-visible points around it. More precisely, for $(r,s)\in L$ define \[\Lambda(r,s)=|\{(n,m)\in L\,|\, \text{$(n,m)$ is $b$-visible and $|n-r|+|m-s|=1$ }\}|,\] and let $M(\Lambda)$ be as in \eqref{mean-value-definition}. Then
	\[
	M(\Lambda)=\frac{4}{\zeta(b+1)}.
	\]
\end{theorem}
\begin{proof}
	Let $\Theta(r,s)=\floor*{\frac{1}{\text{gcd}_b(r,s)}}$ for $(r,s)\in L$. Then
	\[
	\Lambda(r,s)=\sum_{\substack{(n,m)\in L\\|n-r|+|n-s|=1}} \Theta(n,m).
	\]
	For an integer $N>2$ we have that the sum
	\[
	\sum_{0<r,s\leq N}\Lambda(r,s)
	\]
	equals
	\begin{align*}
	4\sum_{0<r,s\leq N} \Theta(r,s)&-\big[\,\Theta(1,1)+\Theta(1,N)+\Theta(N,1)+\Theta(N,N)\,\big]\\
	&-\sum_{i=1}^N\big[\Theta(1,i)+\Theta(i,1)+\Theta(i,N)+\Theta(N,i)\big],
	\end{align*}
	but since $\Theta$ is  a bounded function we clearly have
	\[
	\sum_{0<r,s\leq N}\Lambda(r,s)=4\sum_{0<r,s\leq N} \Theta(r,s) +O(N).
	\]
	The result now follows from Theorem \ref{Intro-Mean-value} applied to $\Theta$ and Remark \ref{harris-theo}.
\end{proof}

Despite the  result above, we will show in Corollary \ref{lonesome points} that the graph $G_b$ is not connected, i.e., for every $b\geq 1$ there are $b$-visible points completely surrounded by $b$-invisible points. The connectivity of the graph $G_1$ was also studied by Vardi  \cite{vardi1999deterministic} in connection with the question of unbounded walks on a single subset of a graph which Vardi calls {\it deterministic percolation}. Vardi shows that there is a unique infinite connected component of $G_1$, denoted by $C_1$, which has an asymptotic density. In particular, Theorem 3.2 and 3.3 of \cite{vardi1999deterministic} shows that the limit \[\theta:=
\lim_{N\to \infty} \frac{|C_1\cap T_N|}{|T_N|} \quad 
\]
exists and it is non-zero, where $T_N$ is  defined in  \eqref{rectangle}. Moreover, his computations seem to indicate that the proportion of $C_1$ in $G_1$ is approximately $0.96\pm.01.$ Therefore $\theta \approx 0.58368$ which experimentally shows that more than $ 58\%$ of lattice points lie in the infinite component.

Since $G_1 \subset G_b$ for $b\geq 2$, the results of \cite{vardi1999deterministic} immediately imply that there is only one infinite connected component of $G_b$, which we denote by $C_b$. Moreover, this infinite connected component has positive density in $G_b$, i.e. there exists a constant $K>0$
such that
\[
K< \frac{|C_b\cap T_N|}{|T_N|} 
\]
for $N\gg 0$. In future work we would like show that the limit 
\[
\lim_{N\to \infty} \frac{|C_b\cap T_N|}{|T_N|}
\]
exists for all $b>1$ and compute it experimentally.

\subsection{Patterns of $b$-visible and $b$-invisible lattice points}
In \cite[Theorem 2]{Harris}   it is shown that the lattice $L$ contains arbitrarily large rectangles containing only $b$-invisible points. This raises the natural question: what other rectangular  arrangements consisting of $b$-visible points and $b$-invisible points  can be found in the lattice $L$? In \cite{Herzog}, Herzog and Stewart gave a complete answer to this question in the case $b=1$. In this section we generalize their work to the case $b\geq 2$.

In order to make the geometrical representations easier to visualize, we will use the same notation as in \cite{Herzog}  and assign a circle ($\circ$) for every $b$-visible point in the lattice  and a cross ($\times$) for every $b$-invisible point.

\begin{definition}\label{b-pattern}
	Let $w$ be a positive integer and to each element $(r,s)\in L$ with $1\leq r\leq w$ and $1\leq s \leq w^b$  assign a cross or a circle or neither. We call this configuration  a  $b$-\textit{pattern} $P$ of $L$. 
\end{definition}

We say that the $b$-pattern $P$ can be \text{realized} in $L$ if there exists a point $(u,v)$ in $L$
such that the rectangle 
\[(u,v)+P=\{\,(r,s)\in L\,:\,u+1\leq r\leq u+w,\ v+1\leq s\leq v+w^b\, \}\]
 has a $b$-visible point whenever $P$ has a circle and a $b$-invisible point whenever $P$ has a cross.

\begin{definition}
	Let $m$ be a positive integer. We call a \textit{complete rectangle modulo}  ${(m,m^b)}$ any collection $S$ of $m^{b+1}$ points in $L$ containing a complete system of residues of the Cartesian product $\Z/m\Z \times \Z/m^b\,\Z$.  
\end{definition}

In what follows, we will use the notation $(x,y)\equiv (r,s)\mod{(m,m^b)}$ to mean that the congruences $x\equiv r \mod{m}$ and $y\equiv s \mod{m^b}$ both hold.

We are now ready to prove the main result in this section.

\begin{theorem}[cf. Theorem \ref{intro-pattern}]\label{patterns}
	A given $b$-pattern $P$ is  realizable in $L$ if and only if the set $C$ of circles in $P$ fails to contain a complete rectangle modulo $(p,p^b)$ for every prime $p$.
\end{theorem}
\begin{proof}
	Let $(u,v)$ be an element in $L$. Assume that the $b$-pattern $P$ is embedded in the square $1\leq r \leq w$, $1\leq s\leq w^b$. Denote by $(u,v)+P$ the translate of every lattice point in the $b$-pattern $P$ by $(u,v)$. If we assume that the $b$-pattern $P$ is such that its set $C$ of circles  contains a complete rectangle modulo $(p,p^b)$ for some prime $p$, then  there   exists an element $(r,s)$ in $(u,v)+P$ for which $(r,s)\equiv (0,0)\mod{(p,p^b)}$. This implies that $p$ divides $\text{gcd}_b(r,s)$,  and thus $(r,s)$ is $b$-invisible. This contradicts that $P$ is realizable in $L$,  which proves the necessity of the condition.
	
	Assume now that the set $C$ of circles in $P$ fails to contain a complete rectangle modulo $(p,p^b)$ for every prime $p$. Then we will find a $(u,v)$ in $L$ such that $(u,v)+P$ contains a $b$-visible point for every circle of $P$ and a $b$-invisible point for every cross in $P$. Such $(u,v)$ will be found as a common solution to three  collections of congruences that we define below.
	
	We define the first collection of congruences as follows. Let $p$ be a prime with $p\leq w$.  The condition of the theorem implies the existence of a point $(r_p,s_p)$ such that $(r,s)\not \equiv (r_p,s_p) \mod{(p,p^b)}$ for all $(r,s)$ in $C$. Let $(u,v)$ be such that
	\begin{equation}\label{moduli-1}
	(u,v)\equiv (-r_p,-s_p)\mod{(p,p^b)}.
	\end{equation}
	For all $(r,s)$ in $C$ we then have $(u+r,v+s)\equiv (r-r_p,s-s_p)\not\equiv (0,0)\mod{(p,p^b)}$, so that $\text{gdc}_b(u+r,v+s)$ is not divisible by $p$. Since the moduli $p$ in \eqref{moduli-1} are relatively prime, we can find a $(u,v)$ so that \eqref{moduli-1} holds simultaneously for all $p\leq w$.
	
	We build now the second collection of congruences. The idea for this collection is to guarantee that every cross in $P$ becomes a $b$-invisible point in $(u,v)+P$. This is done as follows.  To each cross $(i,j)$ in the $b$-pattern $P$ we associate a prime $Q(i,j)>w$, with  different primes $Q(i,j)$
	corresponding to  different points $(i,j)$. To the congruences \eqref{moduli-1} we attach the congruences
	\begin{equation}\label{moduli-2}
	(u,v)\equiv (-i,-j)\mod{(Q(i,j),Q(i,j)^b)},
	\end{equation}
	for each cross $(i,j)$ in the $b$-pattern $P$. The congruence \eqref{moduli-2} 
	implies $(u+i,v+j)\equiv (0,0)\mod{(Q(i,j),Q(i,j)^b)}$. This implies that $Q(i,j)$ divides $\text{gcd}_b(u+i,v+j)$, so that $(u+i,v+j)$ is $b$-invisible for every cross $(i,j)$ in the $b$-pattern $P$. Once again, the Chinese Remainder Theorem guarantees the existence of a common solution $(u,v)$ to \eqref{moduli-1} and \eqref{moduli-2}.
	
	Observe, additionally, that  for this common solution $(u,v)$ we have that the congruence $(u+r,v+s)\equiv (0,0) \mod{(Q(i,j),Q(i,j)^b)}$, for $(r,s)$ with $1\leq r \leq w$ and $1\leq s \leq w^b$, has a solution if and only  if $(r,s)$  coincides with the cross $(i,j)$ in $P$. This is a consequence of the inequalities $Q(i,j)>w$ and  $Q(i,j)^b>w^b$.
	
	The above considerations imply so far that for  a circle $(r,s)$ in $P$ the number $\text{gcd}_b(u+r,v+s)$ is not divisible by the primes $p\leq w$ and $Q(i,j)$. However, it may still happen that $\text{gcd}_b(u+r,v+s)>1$ for some circle $(r,s)$ in $P$. We can remedy this by considering  a third collection of congruences as follows.
	
	First, fix  a positive $u$ satisfying both  \eqref{moduli-1} and \eqref{moduli-2}. The positive numbers $u+1,\dots, u+w$ have a finite number of prime factors which, by the above considerations, are all different than the primes $p\leq w$ and $Q(i,j)$; we use $q$ to denote these prime factors. For each one of these primes $q$ we attach to  \eqref{moduli-1} and \eqref{moduli-2} a new set of congruences
	\begin{equation}
	v\equiv 0 \mod{q},
	\end{equation}
	which has a simultaneous solution by the Chinese Remainder Theorem. Moreover, 
	since $q>w$ (and so $q^b>w^b$) we have that $v+1,\dots, v+w^b$ lies between two multiples of $q^b$, namely $v$ and $v+q^b$, therefore $v+s$ is not divisible by $q^b$ for $1\leq s \leq w^b$. In this way, for every circle $(r,s)$ in $C$ we have that $\text{gcd}_b(u+r,v+s)$ is not divisible by any of the primes $q$. In conclusion, we have that $\text{gcd}_b(u+r,v+s)=1$, i.e., $(u+r,v+s)$ is $b$-visible for every circle $(r,s)$ in $C$. This finishes the proof of the theorem.
	
\end{proof}

It is worth mentioning that since the criterion for a $b$-pattern $P$ to be realizable in $L$ is based on a collection of congruences,  it  immediately follows  that if  $P$ is realizable once then it is realizable infinitely many times.

We finish by  stating a collection of results that are consequences of  Theorem \ref{patterns}.

\begin{corollary}[{ \cite[Theorem 2]{Harris}} ]
	Any $b$-pattern $P$ containing only crosses is realizable in $L$, that is: $L$ has arbitrarily large  $b$-invisible forests. 
\end{corollary}

\begin{corollary}
	Let $P$ be the $b$-pattern consisting of a square with vertices $(1,1)$, $(N,1)$, $(N,N)$ and $(1,N)$, $N\geq 1$, containing only circles. 
	Then $P$ is realizable if and only if $N^2< 2^b$.
\end{corollary}

\begin{corollary}\label{lonesome points}
	Any $b$-pattern $P$ composed of crosses and only one circle is realizable in $L$, that is, there are extremely lonesome $b$-visible points. Therefore the graph $G_b$ defined above is not connected.
\end{corollary}

For example, the point $(6001645,49747967748324)$ has $\text{gcd}_2=1$, but the points around it which are $(6001645+i,49747967748324+j)$ with $(i,j)=(-1,-1),(-1,0),$ $(-1,1),(0,-1),(0,1),(1,-1),(1,0),(1,1)$ have $\gcd_2=19, 6, 11, 13, 5, 17, 2, 7$ respectively.

\begin{corollary}
	Let $P$ be the $b$-pattern that consists of a rectangle with vertices $(1,1)$, $(M,1)$, $(M,N)$ and $(1,N)$, $M\geq 2$, $N\geq 2$, with all of its boundary points being circles and all its interior points being crosses. For $b=1$, we have that $P$ is realizable in $L$ if and only if $M$ and $N$ are both odd (cf. \cite[Corollary 3]{Herzog}). For $b\geq 2$, $P$ is realizable in $L$ if and only if $M$ is odd or $N\geq 2^b$. In particular, there are arbitrarily large rectangular $b$-invisible forests fenced off by $b$-visible points.
\end{corollary}
\begin{proof}
We will assume that $b\geq 2$, since the case $b=1$ can be found in \cite[Corollary 3]{Herzog}.

Let $p>2$ be a prime number. Take $z\mod{p}$ such that $z\not\equiv 1,M,N\mod{p}$. Then $(z,z)\, \text{mod}(p,p^b)$ cannot be congruent to any of the elements in the boundary of $P$ which is described by the set
\[
C:=\{(1,s),\,(r,1),\,(M,s),\,(r,N)\,:\,1\leq r\leq M,\, 1\leq s\leq N\}.
\]
Thus we have shown that $C$ fails to contain a complete rectangle modulo $(p,p^b)$ for $p>2$.
Therefore, for this specific pattern $P$  we have that
\begin{equation}\label{equivalency}
\text{$P$ is realizable in $L$}
\Longleftrightarrow 
\text{$C$  fails to contain a complete rectangle mod($2,2^b$).}
\end{equation}

With this new equivalency in mind we  now proceed to prove the result. Suppose  that $P$ is realizable in $L$. Let us show that either $M$ is odd or $N<2^b$. Suppose not, i.e., $M$
is even and $N\geq 2^b$. Then the following points of $C$
\[
(1,1),\,(1,2),\dots,\,(1,2^b)
\quad \mathrm{and}\quad
(M,1),\,(M,2),\dots,\,(M,2^b)
\]
contain a complete rectangle mod($2,2^b$). This is a contradiction according to \eqref{equivalency}.

Conversely, suppose that $M$ is odd or $N<2^b$. Let us show that $P$ is realizable in $L$. According to 
\eqref{equivalency} it is enough to show $C$ fails to contain a complete rectangle mod($2,2^b$). In order to do this, we will show that it is impossible for the set $C$ to contain all  of the elements
\begin{equation}\label{rectangl}
(2,2), (2,3),\dots, (2,2^b)\quad \mathrm{mod}(2,2^b).
\end{equation} 
from a complete rectangle modulo $(2,2^b)$.
Indeed, if $M$ is odd then only the points from $C$ given by $(r,N)$, $1\leq r\leq M$, could contain all of
 \eqref{rectangl}, but this is impossible as their second component is $N$ which is fixed; recall that $2<2^b$ since we are assuming $b\geq 2$. 

Finally, if $N<2^b$ then none of the points in $C$ is congruent to the pair $(2,2^b)$ mod($2,2^b$) as they all have second component between 1 and $N$ ($<2^b$). 
\end{proof}

\noindent\emph{\bf Acknowledgement.} The authors are grateful to the referees for useful comments and suggestions that
have improved the quality of the manuscript. The second author would like to thank David Nacin for his help with implementing $\gcd_b$ in SAGE.

\end{document}